\RequirePackage{fix-cm}
\documentclass[smallextended]{svjour3}       

\smartqed  
\usepackage{graphicx}

\usepackage{amssymb,latexsym,amsmath}
\usepackage{mathptmx}

\usepackage{amsmath}
\usepackage{amssymb,latexsym}
\usepackage{multicol}

\newcommand{\numberset}{\mathbb}

\newcommand{\F}{\numberset{F}}

\newcommand{\Pro}{\numberset{P}}

\newcommand{\Ol}{\mathcal{O}}

\newtheorem{notation}[theorem]{Notation}
\usepackage{hyperref}

\newtheorem{cexample}[theorem]{Example}

\begin{document}

\title{On the dual minimum distance and minimum weight of codes 
from a quotient of the Hermitian curve}

\titlerunning{Codes from a quotient of the Hermitian curve}

\author{Edoardo Ballico         \and
        Alberto Ravagnani 
}

\institute{Edoardo Ballico \at
              Department of Mathematics \\
              University of Trento\\
              Via Sommarive 14, 38123 Povo, Trento (Italy) \\
              \email{edoardo.ballico@unitn.it}   \and
Alberto Ravagnani \at
              Institut de Math\'{e}matiques \\
              Universit\'{e} de Neuch\^{a}tel\\
              Rue Emile-Argand 11, CH-2000 Neuch\^{a}tel (Switzerland) \\
              \email{alberto.ravagnani@unine.ch}
}

\maketitle

%
%
%
%
%
%

\begin{abstract}
 In this paper we study evaluation codes arising from plane quotients of the 
Hermitian curve, defined by affine equations of the form $y^q+y=x^m$, $q$ being
a prime power and
$m$ a positive integer which divides $q+1$. The  dual minimum distance and 
minimum weight of
such codes are studied from a geometric point of view. 
In many cases we completely describe the minimum-weight codewords of
their dual codes through a
geometric
characterization of the supports, and provide their number. 
Finally, we apply our results to describe Goppa codes of classical interest on
such curves.
\end{abstract}




\setcounter{section}{-1}

\section{Introduction}

Let $q$ be a prime power, and let $X$ be a smooth projective curve
defined over the finite field $\F_q$ with $q$ elements. Consider a divisor $D$
on
$X$, and take points $P_1,...,P_n \in X(\F_q)$ avoiding the support of
$D$. Set $\overline{D}:=\sum_{i=1}^n P_i$.   
The \textbf{Goppa code} $\mathcal{C}(\overline{D},D)$ is defined as the linear
code obtained
evaluating the Riemann-Roch space 
space $L(D)$ at the points $P_1,...,P_n$ (see \cite{Ste} for a geometric
introduction to Goppa codes). If $D$ is supported by $s$ points, then
$\mathcal{C}(\overline{D},D)$ is said to be an $s$-point code. Goppa codes
yield good parameters (see \cite{LS} for a useful analysis) and their minimum
distance can be lower-bounded thanks to their geometric structure (\cite{Ste},
Chapter 10).

The parameters of a Goppa code strictly depend on the curve chosen in the
construction: different curves give different codes. 
The most studied Goppa
codes are probably those arising from the Hermitian curve (see the
references in Subsection \ref{sub1}). In this paper we focus on plane quotients
of the Hermitian curve, and study a wide class of evaluation codes on such
singular curves by means of geometric techniques.
We also show that our family includes many Goppa codes of classical interest,
such as
one-point and two-point codes arising from these curves.

\subsection{Main references on codes from the Hermitian curve and its quotients}
\label{sub1}

One-point codes on a Hermitian curve (see \cite{Sti}, Example VI.3.6)
are well-studied in the literature, and efficient methods to decode them are known
(see for instance \cite{Sti}, \cite{yk} and \cite{yks}).
The minimum distance of Hermitian two-point codes has been first determined by
M. Homma and S. J. Kim
 (\cite{hk1}, \cite{hk2}, \cite{hk3}, \cite{hk4}). S. Park
gave
explicit formulas
for the dual minimum distance of such codes (see \cite{Park}). More recently, 
Hermitian codes from higher-degree places have been considered in \cite{Kor}.
The dual minimum distance of many three-point codes on the Hermitian curve is
computed in \cite{br}, by extending a recent and powerful approach by A.
Couvreur (see \cite{c3}).

Two-point codes arising from quotients of the Hermitian curve
are deeply studied in \cite{clp} and \cite{col}, computing, in particular, their
dimensions. See \cite{gv} and \cite{matth1} for the explicit
computation of several Weierstrass semigroups on these curves.

\subsection{Layout of the paper}

Let us briefly discuss the structure of the paper. We introduce plane
quotients of the Hermitian curve in Section \ref{intr}, summarizing
their projective geometry. In Section \ref{prelim} we define two families of
evaluation codes arising from these curves, and state some
preliminary geometric lemmas. The first family (which we call
\textit{uncomplete codes}) is the simplest one, and it is studied in
depth in Section \ref{uncomplete_codes}. The analysis of the second family
(namely, that of \textit{complete codes}) is
performed in Section \ref{vuoto} and Section \ref{non_vuoto}.
Applications to the study of classical Goppa codes are given in Section
\ref{corollari}.

\section{Preliminaries}
\label{intr}

Let $q$ be a prime power and let $\Pro^2$ denote the projective plane of
homogeneous coordinates $(x:y:z)$ over the
field 
$\F_{q^2}$. Choose a positive divisor 
of $q+1$, say $m$, and consider the curve $Y_m$ defined over $\F_{q^2}$ by the
affine 
equation $y^q+y=x^m$. If $m=q-1$, then $Y_{m}$ is the well-known Hermitian
curve. Here we
focus on the more complicated case $m \ne
q+1$. In this situation  $Y_m$ has exactly one point at infinity,
namely,
$P_\infty=(1:0:0)$. We notice that $Y_m$ is a singular plane curve, carrying
$P_\infty$ as a unique singular point.  
It is well-known that the normalization of $Y_m$, say $C_m$, is a maximal curve.
Moreover, $Y_m$ carries
$|Y_m(\F_{q^2})|=1+q(1+(q-1)m)$ points which are $\F_{q^2}$-rational, and its
geometric genus  
(by definition, the geometric genus of $C_m$) is $(q-1)(m-1)/2$ 
(see \cite{Sti}, Example 6.4.2 at page 234). Denote by $\pi:C_m \to Y_m$ the
normalization of $Y_m$, and let $i:Y_m
\hookrightarrow \Pro^2$ be the inclusion of $Y_m$ into the projective plane.
Since
$\pi$ is injective, the composition $u:= i \circ \pi:C_m \to \Pro^2$ is an
injective morphism. We define $Q_\infty \in C$ by $\pi(Q_\infty)=P_\infty$.
By \cite{Sti}, Proposition 6.4.1, for any integer $r \ge 0$ the
monomials $x^iy^j$ 
such that 
\begin{eqnarray} \label{basis}
 i \ge 0, \ \ \  0 \le j \le q-1, \ \ \  qi+mj \le r
\end{eqnarray}
form a basis of the vector space $L(rQ_\infty)$, the Riemann-Roch space
associated to the divisor $rQ_\infty$ on $C_m$.

\begin{notation}
 In the sequel, we work with a fixed $q$ and a fixed $m$. Hence,
we will always write $Y$ and $C$ instead of $Y_m$ and $C_m$, respectively.
\end{notation}

\section{Definitions and preliminary results}\label{prelim}
Here we introduce two classes of evaluation codes on a quotient
$Y=Y_m$ of the Hermitian curve, and discuss their basic properties.

\begin{definition}\label{defcodici}
Let $E\subseteq Y$ be a zero-dimensional scheme defined over $\F_q$ (the case
$E=\emptyset$ may of course be considered). Fix an integer $d>0$ and
set\footnote{Here the subscript \textit{red} denotes
the reduction of a zero-dimensional scheme.}
$B:=Y(\mathbb {F}_{q^2})\setminus (E_{red} \cup \{P_\infty \})$. The
\textbf{uncomplete} code $\mathcal{B}(d,-E)$ is the code obtained
evaluating the vector space $H^0(\Pro^2,\mathcal{I}_E(d))$ on the set $B$. The
\textbf{complete} code $\mathcal{C}(d,-E)$ is the code obtained evaluating
the vector space $H^0(C,\pi^{*}(\Ol_Y(d))(\pi^{-1}(-E)))$ on $\pi^{-1}(B)$. 
\end{definition}

The aim of this paper is to study the minimum distance and the minimum-weight
codewords of codes of type $\mathcal{B}(d,-E)^\perp$ and
$\mathcal{C}(d,-E)^\perp$.  In Section \ref{corollari} we will show that many
codes on $Y$ of classical
interest (such as Goppa one-point and two-point codes) can be easily studied as
$\mathcal{C}(d,-E)$ codes.

\begin{proposition}
 Let $\mathcal{B}(d,-E)$ and $\mathcal{C}(d,-E)$ be as in Definition
\ref{defcodici}. If $d<q$, then $\mathcal{B}(d,-E)$ is a subcode of
$\mathcal{C}(d,-E)$. Moreover, the minimum distance of
$\mathcal{C}(d,-E)^\perp$ is at
least the minimum distance of $\mathcal{B}(d,-E)^\perp$.
\end{proposition}
\begin{proof}
 Since (by assumption) $d<q$, the restriction (and pull-back) map of cohomology
groups
\begin{equation*}
 \rho_{d,E}:H^0(\Pro^2,\mathcal{I}_E(d)) \to
H^0(C,\pi^*(\Ol_Y(d))(\pi^{-1}(-E)))
\end{equation*}
is injective. It follows
$\mathcal{B}(d,-E) \subseteq \mathcal{C}(d,-E)$, and so $\mathcal{B}(d,-E)^\perp
\supseteq \mathcal{C}(d,-E)^\perp$. Since the minimum distance of a code is
computed by taking the minimum of the pairwise Hamming distances of the
codewords, the second part of the statement easily follows.
\end{proof}

The next result states technical properties of
zero-dimensional
subschemes of the projective plane $\Pro^2$. Lemma \ref{u00.01} is a key-point in our
approach, providing a geometric
interpretation to certain non-vanishing conditions of cohomology groups. We
will use this result many times throughout the paper to get necessary conditions
on the supports of some minimum-weight codewords.

\begin{lemma}\label{u00.01}
Fix integers $d>0$, $z\ge 2$ and a zero-dimensional scheme $Z\subseteq \mathbb
{P}^2$ such that
$\deg (Z) =z$.
\begin{enumerate}
\item[(a)] If $z\le d+1$, then $h^1(\mathbb {P}^2,\mathcal {I}_Z(d))=0$.

\item[(b)]\label{strano1} If $d+2 \le z\le 2d+1$, then $h^1(\mathbb
{P}^2,\mathcal {I}_Z(d))>0$ if and only if there
exists a line $L \subseteq \Pro^2$ such that $\deg (L\cap Z)\ge d+2$.
\end{enumerate}
\end{lemma}

\begin{proof}
 See \cite{bgi}, Lemma 34.
\end{proof}

 It is fundamental to our purpose to state Lemma \ref{u00.01} for
zero-dimensional schemes $Z \subseteq \Pro^2$, and not just for finite sets.
Moreover, the absence of such a lemma in higher-dimensional projective spaces
 is the main reason why we are forced to work in the plane, instead of on the
normalizations of $Y_m$ curves directly.

\section{Uncomplete codes}\label{uncomplete_codes}

Here we study the geometric properties of uncomplete $\mathcal{B}(d,-E)$ codes,
find out their dual minimum distance, and characterize the minimum-weight
codewords of their dual codes. We notice that the length of a
$\mathcal{B}(d,-E)$ code is
$1+q+q^2m-qm- |{E_{red} \cup \{Q_\infty\} }|$. Indeed, there
is no non-zero
global
section of $\mathcal{I}_E(d)$ vanishing at all the points of $B$
(in the notation of Definition \ref{defcodici}).

\begin{lemma} \label{h1}
 Let $B$ and $\mathcal{B}(d,-E)$ be as in Definition \ref{defcodici}. Assume
$\deg(E) \le d+1$. Let $S \subseteq B$ be a subset. There exists a codeword of
$\mathcal {B}(d,-E)^\perp$ whose support is contained in $S$ if and only if
$h^1(\Pro^2,\mathcal{I}_{E\cup S}(d))>0$. 
\end{lemma}
\begin{proof}
 First of all, we prove that $h^1(\Pro^2,\mathcal{I}_E(d))=0$. By
contradiction, let us assume $h^1(\Pro^2,\mathcal{I}_E(d))>0$. By Lemma
\ref{u00.01}, there exists
a line $L \subseteq \Pro^2$ such that $\deg(L\cap E) \ge d+2$. This
contradicts $\deg(E) \le d+1$. Recall that the code $\mathcal {B}(d,-E)$
is obtained by evaluating on $B$ all the degree $d$ homogeneous forms vanishing
on $E$. A subset $S \subseteq B$ contains the support of a minimum-weight
codeword of $\mathcal {B}(d,-E)^\perp$ if and only if it imposes dependent
conditions to $H^0(\Pro^2, \mathcal{I}_E(d))$. Since $E\cap S=\emptyset$, this
is
equivalent to say that $h^1(\Pro^2,\mathcal{I}_{E\cup S}(d))>
h^1(\Pro^2,\mathcal{I}_{E}(d))=0$.
\end{proof}

\begin{proposition}\label{a1}
Let $\mathcal{B}(d,-E)$ be as in Definition \ref{defcodici}.
Assume $d \le m-2$
and $\deg(E) \le d-1$. Let ${\bf {w}}$ be  a codeword  of $\mathcal
{B}(d,-E)^{\bot}$, and let $S$ denote the support of ${\bf {w}}$. Assume $\deg(E
\cup
S) \le 2d+1$. The following facts hold.

\begin{enumerate}
 \item There exists a line $L\subseteq \mathbb {P}^2$, defined over $\F_{q^2}$,
with $S\subseteq L$ and  $\deg ((E\cup S)\cap L) \ge d+2$. 

\item Any $S'\subseteq S\cap L$ with $|S'| = d+2 -\deg (E\cap L)$
is the support of a codeword of $\mathcal {B}(d,-E)^{\bot}$ of weight
$d+2-\deg (E\cap L)$.

\item If
${\bf {w}}$ is a minimum-weight codeword of $\mathcal
{B}(d,-E)^{\bot}$,
then $|S| = d+2 -\deg (E\cap L)$.
\end{enumerate}
\end{proposition}

\begin{proof}
\begin{enumerate}
 \item By Lemma \ref{h1}, we have $h^1(\mathcal {I}_{E\cup S}(d)) >0$.
Since $\deg
(E\cup S) \le 2d+1$, by Lemma \ref{u00.01} there exists a
line
$L \subseteq \mathbb {P}^2$ such that $\deg (L \cap (E\cup S)) \ge d+2$.
Since $\deg (E)\le d-1$, we have $|S\cap L|\ge 3$. Since any point of $S$
is defined over $\mathbb {F}_{q^2}$, $L$ is defined over $\mathbb {F}_{q^2}$
itself.
\item Let $S_1\subseteq S\cap L$ be any subset with cardinality
at least $d+2 -\deg (E\cap L)$. We may take, for instance, $S_1=S\cap L$.
Since $\deg (L\cap (S_1\cup E)) \ge d+2$, we have
 $h^1(\mathcal {I}_{S_1\cup E}(d)) >0$ (Lemma \ref{u00.01} again). By Lemma
\ref{h1}, there exists a subset $S'\subseteq S_1$ which is the support of a
codeword ${\bf {w}}'$ of $\mathcal {B}(d,-E)^{\bot}$.
\item Apply part (2) and the definition of minimum distance.
\end{enumerate}
\end{proof}

\section{Complete codes: the case $E=\emptyset$}\label{vuoto}

This section and the following one are devoted to the more interesting class of
complete codes $\mathcal{C}(d,-E)$. When $E=\emptyset$, we will simply
write $\mathcal{C}(d)$ instead of $\mathcal{C}(d,\emptyset)$. We begin our
analysis by studying $\mathcal{C}(d)$ codes, whose geometry is rather simple.
In Section \ref{non_vuoto} we will consider the more general case, and extend
our results. 

\begin{definition}
 A line $L \subseteq \Pro^2$ is said to be \textbf{horizontal} if it has an equation of
the form $y=a$, for a certain $a \in \F_{q^2}$. The line at infinity,
of equation $z=0$, will be denoted by $L_\infty$. Notice that we do not consider this line horizontal. 
\end{definition}

\begin{lemma}\label{base}
 Let $\mathcal{C}(d)$ be a complete code with $d \le
m-2$. Then the minimum distance of $\mathcal{C}(d)^\perp$ is $d+2$, and $d+2$ points
in the support of a minimum weight codeword of $\mathcal{C}(d)^\perp$ are
collinear.
\end{lemma}

\begin{proof}
 Since $d \le m-2$, there exist $d+2$ points lying on the intersection of
$Y(\F_{q^2})\setminus \{ P_\infty \}$ and a horizontal line of the form $y=a$,
with $a \in \F_{q^2}\setminus \{ 0\}$. Such $d+2$ points contain the support of a codeword of
$\mathcal{C}(d)^\perp$, and prove that the minimum distance of
$\mathcal{C}(d)^\perp$ is at most $d+2$. It remains to be shown that the minimum distance of
 $\mathcal{C}(d)^\perp$ is at least $d+2$.
Consider the
restriction (and pull-back) map
\begin{equation*}
 \rho_d:H^0(\Pro^2,\Ol_{\Pro^2}(d)) \to H^0(C,\pi^*(\Ol_Y(d))).
\end{equation*}
Since $d<q$, the map $\rho_d$ is injective. Let $S \subseteq
C(\F_{q^2})\setminus \{ Q_\infty \}$ be a finite subset of points.
Since $\pi^{-1}(P_\infty)=\{ Q_\infty \}$, we have $|S|=|\pi(S)|$. If
the set $S$ imposes dependent conditions to $H^0(C,\pi^*(\Ol_Y))$ (i.e., if
$\pi(S)$ contains the support of a non-zero codeword of $\mathcal{C}(d)^\perp$)
then it imposes dependent conditions also to $\mbox{Im}(\rho_d)$. By
the injectivity of $\rho_d$, we have that $\pi(S)$ imposes dependent conditions
to $H^0(\Pro^2, \Ol_{\Pro^2}(d))$. In other words,
$h^1(\Pro^2,\mathcal{I}_{\pi(S)}(d))>0$. Assume that $S$ is the support of a
minimum-weight codeword of $\mathcal{C}(d)^\perp$. We have
$|S| \le d+2$ (see the first part of the proof). Moreover, the weight of the codeword is exactly
$|S|$. Lemma \ref{u00.01} gives $|S| \ge d+2$ and that $S$ is
contained into a line of $\Pro^2$. This proves that the minimum distance of
$\mathcal{C}(d)$ is exactly $d+2$, and $d+2$ points of the support of a
minimum-weight codeword must be collinear.
\end{proof}

If we drop the assumption $d\le m-2$, then the proof of Lemma
\ref{base}
does not work. In fact, the result is not always true when $d>m-2$ (see
Example
\ref{contro}).

\begin{cexample}\label{contro}
 Set $q:=5$, $m:=2$ and $d:=1$. By writing a simple \texttt{MAGMA} program (the
source code is similar to that of Example \ref{exok} below) it can be seen that
the minimum distance of $\mathcal{C}(d)^\perp$ is $4 \neq d+2$. Let
$\F_{25}=\langle a \rangle$, with $a^2+4a+2=0$. Then there exists a
minimum-weight codeword of $\mathcal{C}(d)^\perp$ whose support consists of the
following four points.
$$(a^3,a^{11}), \ \ (a^{21},a^{22}), \ \ (a^9,a^{23}), \ \  (0,0).$$
No three points of them are collinear.
\end{cexample}

Lemma \ref{base} provides the dual minimum distance of
$\mathcal{C}(d)$ codes, and gives necessary conditions for a set to be the
support of a minimum-weight codeword of $\mathcal{C}(d)^\perp$. It is easily
checked that any $d+2$ points lying on the intersection
of $Y$ and a horizontal line are the support of a minimum-weight codeword of
$\mathcal{C}(d)^\perp$. We are going to show that if $(q+1)/m \ge 3$, then this
condition characterize the supports of the minimum-weight codewords of a
$\mathcal{C}(d)^\perp$ code. This gives, in particular, the exact number of the
minimum-weight codewords.

\begin{remark} \label{proj}
 Let $i:Y \hookrightarrow \Pro^2$ be the inclusion of $Y$ in the projective
plane, $\pi: C \to Y$ the normalization  and $u:=C \to \Pro^2$ the morphism
defined by $u:=i \circ \pi$. The vector space $V \subseteq L(qQ_\infty)$ spanned
by $\{ 1,x,y \}$ is the linear system on $C$ (via pull-back through $\pi$)
inducing $u$. Set $c:=(q+1)/m$ and observe that the conditions of (\ref{basis})
in Section \ref{intr} give
\begin{eqnarray}
 \dim L(qQ_\infty)=h^0(C,\pi^*(\Ol_Y(1)))=c+1.
\end{eqnarray}
Moreover, $\{ 1,x,y,...,y^{c-1} \}$ form a basis of the vector space
$L(qQ_\infty)$. Assume $c \ge 3$, i.e., $V \varsubsetneq L(qQ_\infty)$, and
consider the linear system $W \subseteq L(qQ_\infty)$ spanned by $\{ 1,x,y,y^2
\}$. Since $W\supseteq V$ and $V$ has no base points, $W$
has no base points itself. Hence, it induces a morphism $v: C \to \mathbb {P}^3$.
Since $u$ is injective
and $W\supseteq V$, also $v$ is injective. Since $u$ has non-zero-differential
at each point of $C\setminus \{Q_\infty \}$, $v$ has non-zero differential at
each point of $C\setminus \{Q_\infty \}$. The curve $T:= v(C) \subseteq \Pro^3$
is non-degenerate, $v: C \to T$ is injective and an isomorphism, except at most
at $v(Q_\infty)$. Let $\ell _{v(Q_\infty )} : \mathbb {P}^3\setminus
\{v(Q_\infty)\} \to \mathbb {P}^2$ denote the linear projection from $v(Q_\infty
)$. Since $u(Q_\infty)=(1:0:0)$, the rational map $\ell_{v(Q_\infty )}: T
\dasharrow \Pro^2$ is induced by $(1,y,y^2)$, and the image $\ell _{v(Q_\infty
)}(T\setminus \{v(Q_\infty )\})$ is contained
into a plane conic $E$. Hence, $T$ is contained in a quadric cone $\Gamma$ with
$v(Q_\infty)$ as its vertex. We observe that the fibers of the rational map
$\ell_{v(Q_\infty )}:T \dasharrow E$ are (outside $Q_\infty$) the elements of
$\vert mQ_\infty\vert$.
Take any line $L \subseteq \mathbb {P}^3$ such that $|L \cap v(C)| \ge 3$.
Bezout theorem gives $L \subseteq \Gamma$. Hence, $v(Q_\infty )\in L$ and
$v(C)\cap (L \setminus \{v(Q_\infty \})$ is an element of $\vert mQ_\infty
\vert$. 
\end{remark}

Remark \ref{proj} proves in fact the following result.
\begin{lemma}\label{lemmatecn}
Take the set up of Remark \ref{proj}. Let
$P_1,P_2,P_3 \in C(\F_{q^2})\setminus \{ Q_\infty \}$ be any three distinct
points
such that $v(P_1)$, $v(P_2)$ and $v(P_3)$ are collinear. Then
$u(P_1),u(P_2),u(P_3)$
lie on an horizontal line of $\Pro^2$.
\end{lemma}

\begin{theorem}\label{a3}
Let $\mathcal{C}(d)$ be a complete code. Assume $d \le m-2$ and
$c=(q+1)/m\ge
3$. The following facts hold.
\begin{enumerate}
 \item The minimum distance of $\mathcal {C}(d)^\perp$ is $d+2$.
\item A set $S\subseteq C(\mathbb {F}_{q^2})\setminus \{Q_\infty \}$ is the
support of a minimum-weight codeword of $\mathcal {C}(d)^\perp$ if and only if
$\pi(S)$ consists of $d+2$ points lying on a horizontal line.
\item The number of the minimum-weight codewords of $\mathcal{C}(d)^\perp$ is
$$(q-1)(q^2-1)\binom{m}{d+2}.$$
\end{enumerate}
\end{theorem}
\begin{proof}
 Part (1) is a particular case of Lemma \ref{base}. Now we prove part (2). By
Lemma \ref{base}, the minimum distance of $\mathcal{C}(d)^\perp$ is $d+2$, and
$d+2$ points in the support of a minimum-weight codeword must be
collinear.
Assume that $S:=\{P_1,P_2,P_3,...,P_{d+2}\}$ is a set of $d+2$ points 
contained into a line
which is not horizontal. If $S$ is the support of a codeword of
$\mathcal{C}(d)^\perp$, then it is the support of a minimum-weight codeword. By
Lemma \ref{u00.01}, the points $P_4,...,P_{d+2}$ impose independent conditions
to the vector space
$H^0(\Pro^2,\Ol_{\Pro^2}(d-1))$. Moreover, there exists a degree $d-1$ plane
curve, say $X$, which contains $P_4,...,P_{d+2}$, but  contains neither
$P_1$, nor $P_2$, nor $P_3$. Hence, it is enough to show that $P_1$, $P_2$ and
$P_3$
impose independent conditions to $H^0(C,\pi^*(\Ol_Y(1))$. This 
follows from
Remark \ref{proj} and Lemma \ref{lemmatecn}. To get the third part, 
observe that the minimum-weight codewords of any code having a
fixed support form a line of the code (by definition of minimum
distance).
\end{proof}

If we drop the assumption $c \ge 3$, then the proof of Lemma
\ref{lemmatecn} does not work. In fact, Theorem \ref{a3} is not true in general when $c=2$
(see Example \ref{exok}).

\begin{cexample} \label{exok}In our usual notation, choose $q:=5$, $m:=3$ and
$d:=1$. It follows $c=(q+1)/m=2$. Let $\mathcal{C}(1)$ be the code obtained
evaluating  $H^0(C,\pi^*(\Ol_Y(1)))$ on the set $\pi^{-1}(B)$. A
basis of this vector space is $\{ 1,x,y\}$. The following \texttt{MAGMA} program
constructs the code $\mathcal{C}(1)^\perp$, prints a random minimum-weight
codeword and its support.

\begin{multicols}{2}
\footnotesize{
\begin{verbatim}
  q:=5;
  m:=3;
  F<a>:=GF(q^2);       
  A<x,y>:=AffineSpace(F,2);
  f:=y^q+y-x^m;
  X:=Curve(A,f);
  pts:=Points(X);
  npts:=#pts;
  P<u,v>:=PolynomialRing(F,2);
  B:=[ [0,0], [0,1],  [1,0]];
  nf:=3;
  rows:=[];
  for i in [1..nf] do;
  Append(~rows,[]);
  end for;
  for j in [1..nf] do;
  for i in [1..npts] do;
  Append(~rows[j], Evaluate(
  u^B[j][1]*v^B[j][2],
  [pts[i][1],pts[i][2]]));
  end for;
  end for;
  RR:=[];
  for i in [1..nf] do; 
  for j in [1..npts] do;      
  RR[npts*(i-1)+j]:=rows[i][j];
  end for;
  end for;
  G:=Matrix(F,nf,npts,RR);
  C:=LinearCode(G);
  D:=Dual(C);
  mww:=MinimumWord(D);
  supp:=[];
  for j in [1..npts] do;
  if mww[j] ne 0 then Append(
  ~supp, [pts[j][1],pts[j][2]]);
  end if;
  end for;
  mww;
  supp;
\end{verbatim}
}
\end{multicols}

On our Linux machine (AMD processor, 32 bits) the output is the
following.

 \

\footnotesize{
\begin{verbatim}
  ( 0 0 0 1 0 0 0 0 0 0 0 0 0 0 0 0 0 0 0 0 0 0 0
    0 0 0 0 0 0 0 0 0 0 0 0 0 0 0 0 0 0 0 0 0 0 0
    0 0 0 0 0 0 0 0 0 0 0 0 0 0 0 0 0 a^14 a^15 )

  [ 1, a ],
  [ a^22, a^23 ],
  [ 0, 0 ]
\end{verbatim}
}

\normalsize

\ 

We see that the points in the support lie on the line of
equation $y=ax$, which is not a horizontal one (here $a$ is a primitive
element of $\F_{25}$, chosen by the program to explicitly construct the field).
\end{cexample}

\section{Complete codes: the general case}\label{non_vuoto}

In this section we discuss the general case of $\mathcal{C}(d,-E)$ codes,
extending our previous results. 

\begin{notation} \label{linee}
 Let $L_0$ be the line in $\Pro^2$ of equation $y=0$. This line intersects $Y$ only at
$(0:0:1)$ with multiplicity $m$. We denote by $\Lambda$ the set of all the plane
horizontal lines different from $L_0$.
 Moreover, we denote by $\Theta$ the set of the other lines of $\Pro^2$ defined
over
$\mathbb {F}_{q^2}$ which do not pass through $P_\infty$. Finally, set
$\alpha _1:=
\max_{L \in \Lambda} \deg (L\cap E)$ and $\alpha _2:= \max_{L \in \Theta} \deg
(L \cap E)$. We recall that the line of equation $z=0$ is denoted by $L_\infty$.
\end{notation}

\begin{remark}
 We notice that $\alpha_1$ is in fact easy to compute. Indeed, the geometry of
$Y$
implies  $\alpha_1=\max _{L \in \Lambda } |L\cap E_{red}|$. Note
also that $|\Lambda|=q-1$.
\end{remark}

\begin{theorem}\label{prult}
 Let $\mathcal{C}(d,-E)$ be a complete code (see Definition
\ref{defcodici}). Assume $\deg(E)
\le d-1$ and $d \le m-2$.
\begin{enumerate}
\item  The minimum distance of $\mathcal{C}(d,-E)^\perp$ is greater or equal
than $d+2 -\max \{\alpha _1, \alpha _2\}$.
\item If $\alpha_1 \ge \alpha_2$, then the minimum distance of
$\mathcal{C}(d,-E)^\perp$ is exactly $d+2-\alpha_1$. Moreover, the number
of the
minimum-weight codewords of $\mathcal{C}(d,-E)^\perp$ is at least
$(q-1)(q^2-1)\binom{m}{d+2-\alpha_1}$.
\item Assume $d \le m-4$, $\alpha _1=\alpha _2$ and $d\ge \alpha _1+1$. Then the
support of any minimum-weight codeword of $\mathcal {C}(d,-E)^\bot$ is contained
into a horizontal line $L\in \Lambda$. Moreover, the number of the
minimum-weight codewords of $\mathcal{C}(d,-E)^\perp$ is exactly
$(q-1)(q^2-1)\binom{m}{d+2-\alpha_1}$.
\end{enumerate} 
\end{theorem}
\begin{proof}
 Since $d<q$, the restriction (and pull-back) map
\begin{equation*}
 \rho_{d,E}:H^0(\Pro^2,\mathcal{I}_E(d)) \to
H^0(C,\pi^*(\Ol_Y(d))(\pi^{-1}(-E)))
\end{equation*}
is injective. Let $S$ be the support of a minimum-weight codeword of $\mathcal
{C}(d,-E)^\bot$. In this situation the weight of the codeword is exactly
$|S|$ (and not only \textit{smaller or equal than} $|S|$). By definition, $S$
imposes dependent conditions to the vector spaces
$H^0(C,\pi^*(\Ol_Y(d))(\pi^{-1}(-E)))$ and 
$H^0(\Pro^2,\mathcal{I}_E(d))$ (here we used the injectivity of
$\rho_{d,E}$). As
in
the proof of Lemma \ref{h1}, we
have $h^1(\Pro^2, \mathcal{I}_{E \cup S}(d))>h^1(\Pro^2, \mathcal{I}_{E}(d))=0$.
By
Lemma \ref{u00.01}, there exists a line $L
\subseteq \Pro^2$ such that $\deg(L \cap(E \cup S)) \ge d+2$. Since $S \cap
E=\emptyset$ and $L$ cannot be $L_0$ nor $L_\infty$, we have
$$d+2 \le \deg(L \cap(E \cup S)) \le \max \{ \alpha_1,\alpha_2\} + |S|.$$
It follows $|S| \ge d+2 - \max \{ \alpha_1,\alpha_2\}$. This proves the first statement.
If $\alpha_1 \ge \alpha_2$ then, by 1., we have that the minimum distance of
$\mathcal{C}(d,-E)^\perp$ is at least $d-2-\alpha_1$. Since, for any $L \in
\Lambda$, $d+2-\alpha_1$
points in $Y(\F_{q^2}) \cap (L\setminus (L \cap E_{red}))$ contain the
support of a codeword of $\mathcal{C}(d,-E)^\perp$, we
have that the minimum distance of $\mathcal{C}(d,-E)^\perp$ is exactly
$d+2-\alpha_1$. The lower bound on the
number of the minimum-weight codewords trivially follows. Now we prove the last statement.
If $\alpha _1=\alpha _2=0$, then $E=\emptyset$ and the thesis is just Lemma
\ref{base}. Hence,
we may assume $\alpha _1>0$. Let $S\subseteq \mathbb {P}^2$ be the image in
$\mathbb {P}^2$ of the support
of a codeword with minimum-weight $d+2-\alpha _1$. By Lemma \ref{u00.01}, there
exists a line $L\in \Lambda \cup
\Theta$ such that $S\subseteq L$. Write $S = \{P_1,\dots ,P _{d+2-\alpha _1}\}$.
As in the
proof of Theorem \ref{a3}, it is enough to find a degree $d-1$ plane
curve
$X\supseteq E\cup \{P_4,\dots ,P_{d+2-\alpha _1}\}$ such that $P_i\notin X$ for
$i=1,2,3$.
Set $A:=  \{P_4,\dots ,P_{d+2-\alpha _1}\}$.
Let $\mbox{Res}_L (A\cup E)$ denote the residual scheme
of $A \cup E$ with respect to $L$, i.e., the closed
subscheme of $\mathbb {P}^2$ with $\mathcal {I}_{A\cup E}:\mathcal {I}_L$ as its
ideal sheaf\footnote{We recall that $\mbox{Res}_L (A \cup E) = \mbox{Res}_L (E)$,
because $A\subseteq
L$,
and $\deg (\mbox{Res}_L (E)) =
\deg (E) -\deg (L \cap E)$.}. We have an exact sequence of sheaves
\begin{equation}\label{eqz1}
0 \to \mathcal {I}_{\mbox{Res}_L (E)}(d-2) \to \mathcal {I}_{A\cup E}(d-1) \to
\mathcal {I}_{A\cup (E\cap L),L}(d-1)\to 0.
\end{equation} 
Observe that $\deg (\mbox{Res}_L (E)) -\alpha _1\le d-2$, and so we use
\cite{bgi},
Lemma 34, in order to compute
$h^1(\mathcal {I}_{\mbox{Res}_L (E)}(d-2))=0$ . Hence, the restriction map
$$\rho : H^0(\mathbb {P}^2,\mathcal {I}_{E\cup A}(d-1)) \to H^0(L,\mathcal
{I}_{A\cup (E\cap L),L}(d-1))$$ turns out to be surjective. Observe that $E$ and
each $P_i$ are defined over $\mathbb {F}_{q^2}$. Take a degree $d-1$ effective
divisor, $F$, defined
over $\mathbb {F}_{q^2}$ and containing $A\cup (E\cap L)$, but not containing
any $P_i$ with $i\le 3$. Let $f$ be an equation of $F$ defined
over $\mathbb {F}_{q^2}$. Take $f_1\in  H^0(\mathbb {P}^2,\mathcal {I}_{E\cup
A}(d-1))$
defined over $\mathbb {F}_{q^2}$ and such that $\rho (f_1)=f$. Finally, choose
$X:=
\{f_1=0\}$.
\end{proof}

\section{Goppa codes of classical interest} \label{corollari}

In this section we apply our analysis to classical problems in geometric Coding
Theory. More precisely, we are going to study Goppa codes arising from
quotients of the Hermitian curve. Since $Y=Y_m$ curves are not
smooth, 
when writing \lq\lq \ Goppa code on $Y$ \rq\rq\ we always mean \lq\lq\ Goppa
code on
$C$ \rq\rq\ 
(the normalization of $Y$). The points of $Y$ will be identified with those of
$C$ through 
the injectivity of the normalization $\pi:C \to Y$ (see Section \ref{intr}).

\begin{definition}\label{si}
 Let $q$ be a prime power and $n \ge 2$ be an integer. We say that codes
$\mathcal{C}, \mathcal{D} \subseteq \F_q^n$ are \textbf{strongly
isometric}
if there exists a vector ${\bf{x}}=(x_1,...,x_n)\in \F_q^n$ of non-zero
components 
such that $$\mathcal{C}={\bf{x}}\mathcal{D}:=
 \{(x_1v_1,...,x_nv_n)\in \F_q^n \ \mbox{ s.t. } \ (v_1,...,v_n) \in \mathcal{D}
\}.$$
The notation will be $\mathcal{C} \sim \mathcal{D}$, and this clearly defines
an equivalence relation on the set of codes in $\F_q^n$.
\end{definition}

\begin{remark}\label{-1}
 Take the setup of Definition \ref{si}. Then $\mathcal{C} \sim \mathcal{D}$
 if and only if $\mathcal{C}^\perp \sim \mathcal{D}^\perp$. Indeed, if
$\mathcal{C}={\bf{x}}\mathcal{D}$
then $\mathcal{C}^\perp={\bf{x}}^{-1}\mathcal{D}^\perp$, where
${\bf{x}}^{-1}:=(x_1^{-1},...,x_n^{-1})$.
A strongly isometry of codes preserves the minimum distance of a code,
its weight distribution
and the supports of its codewords.
\end{remark}

\begin{remark}\label{ct}
 Let $X$ be a smooth projective curve defined over $\F_q$, and let $D$ and
$D'$ be divisors on
$X$. Take points 
$P_1,...,P_n \in X(\F_q)$ 
which do not appear neither in the support of $D$, nor in the support of $D'$.
 Set $\overline{D}:=\sum_{i=0}^n P_i$. It is known (see \cite{mp}, Remark 2.16)
that if 
$D\sim D'$ (as divisors)
 then  $\mathcal{C}(\overline{D},D) \sim \mathcal{C}(\overline{D},D')$.
By Remark \ref{-1}, we have also $\mathcal{C}(\overline{D},D)^\bot \sim
\mathcal{C}(\overline{D},D')^\bot$. Studying Goppa codes up to strong
isometries is a well-established praxis (see \cite{mp} again).
\end{remark}

Here we show that many Goppa codes on $Y$ (more precisely, on $C$)
can be studied, up to strong isometries, by using
 the results of the previous sections.

\begin{remark}\label{equivalenze}
 Let $Y=Y_m$ be a quotient of the Hermitian curve over the finite field
$\F_{q^2}$, and denote by $\pi:C \to Y$ its normalization. The curve $C$ carries
the following identity of vector spaces (see \cite{Sti}, Proposition 6.4.1):
$$L(qQ_\infty)=H^0(C,\pi^*(\Ol_Y(1))).$$ Moreover, since $C$ is maximal, for any
pair of rational points $P,Q \in C(\F_{q^2})$ we get a linear equivalence
$(q+1)P \sim (q+1)Q$ (see for instance \cite{RS}, Lemma 1).
\end{remark}

\begin{corollary}[one-point codes] \label{coroone}
 Let $0 \le r \le (m-2)q$ be an integer. Denote by $\mathcal{C}_r$
the (Goppa) one-point code on $Y$ obtained evaluating the vector space
$L(rP_\infty)$ on the set $B:=Y(\F_{q^2})\setminus \{ P_\infty \}$. Write
$r=dq+e$ with $0 \le e \le q-1$, and assume $0 \le e \le d-1$.
\begin{enumerate}
 \item If $e=0$, then the minimum distance of $\mathcal{C}_r^\perp$ is $d+2$
and the minimum-weight codewords of $\mathcal{C}_r^\perp$ are
at least $(q-1)(q^2-1)\binom{m}{d+2}$. Moreover, if $(q+1)/m \ge 3$, then
minimum-weight codewords of $\mathcal{C}_r^\perp$ are exactly
$(q-1)(q^2-1)\binom{m}{d+2}$.
\item If $e>0$, then the minimum distance of $\mathcal{C}_r^\perp$ is $d+1$ and
the minimum-weight codewords of $\mathcal{C}_r^\perp$ are at least
$(q-1)(q^2-1)\binom{m}{d+1}$.
\end{enumerate}
\end{corollary}
\begin{proof}
 By Remark \ref{equivalenze} and Remark \ref{ct}, we have
$\mathcal{C}_r\sim\mathcal{C}(d,-eP_\infty)$. Now apply Lemma \ref{base},
Theorem \ref{a3} and Theorem \ref{prult}.
\end{proof}

\begin{corollary}[two-point codes] \label{corotwo}
 Let $a,b$ be integers such that $a+b>0$, and let $P \in Y(\F_{q^2})$ be a
rational point different from $P_\infty$. Denote by $\mathcal{C}(a,b,P)$ the
(Goppa) two-point code on $Y$ obtained evaluating the vector space
$L(aP_\infty+bP)$ on the set $B:=Y(\F_{q^2})\setminus \{ P_\infty, P \}$. There
always exist integer $d,a',b'$ such that
\begin{itemize}
 \item[(1)] $d>0$, $a',b' \ge 0$;
\item[(2)] $aP_\infty+bP \sim dqP_\infty-a'P_\infty-b'P$.
\end{itemize}
Assume that these integers $d,a',b'$ can be chosen with the properties
\begin{itemize}
 \item[(3)] $d\le m-2$;
\item[(4)] $0 \le a'+b' \le d-1$;
\item[(5)] $b'>0$.
\end{itemize}
Denote by $P_0 \in Y(\F_{q^2})$ the point of coordinates $(0:0:1)$. The
following facts hold.
\begin{itemize}
 \item[(A)] If $a=0$ and $P=P_0$, then the minimum distance of
$\mathcal{C}(a,b,P)^\perp$ is greater or equal than $d+1$.
\item[(B)] If either $a=0$ and $P \neq P_0$, or $a>0$ and $P=P_0$, then the
minimum distance of $\mathcal{C}(a,b,P)^\perp$ is exactly $d+1$ and the 
minimum-weight codewords of $\mathcal{C}(a,b,P)^\perp$ are at least
$(q-1)(q^2-1)\binom{m}{d+1}$. Moreover, if $d \le m-4$ then the equality holds.
\item[(C)] If $a>0$ and $P \neq P_0$, then the minimum distance of
$\mathcal{C}(a,b,P)^\perp$ is exactly $d$, and the  minimum-weight
codewords of $\mathcal{C}(a,b,P)^\perp$ are at least $(q-1)(q^2-1)\binom{m}{d}$.
\end{itemize}
\end{corollary}

\begin{proof}
 The first part of the proof trivially follows from Remark \ref{equivalenze}.
Set $E:= a'P_\infty+b'P$, and observe that $\mathcal{C}(a,b,P) \sim
\mathcal{C}(d,-E)$. By Remark \ref{ct}, we can apply Theorem \ref{prult}. In the
notation of the cited theorem, in case (A) we have $\alpha_1=0$ and
$\alpha_2=1$, in case (B) we have $\alpha_1=\alpha_2=1$. Finally, in
case (C)
we have $\alpha_1=2$, $\alpha_2=1$.
\end{proof}

\begin{remark}
 Let $\mathcal{C}(\overline{D},D)$ be any non-trivial Goppa code on the curve
$Y$. Let $\{ P_1,...,P_r\}$ be the support of $D$. There always exist integers
$d>0$ and ${(a_i)}_{i=1}^{r}$ such that $a_i \ge 0$ for each $i$ and 
$$D \sim dqP_\infty - \sum_{i=1}^r a_i P_i.$$
By setting $E:=\sum_{i=1}^r a_i P_i$, we have $\mathcal{C}(\overline{D},D) \sim
\mathcal{C}(d,-E)$. Moreover, if $\sum_{i=1}^r a_i \le d-1$ then the code
$\mathcal{C}(\overline{D},D)^\perp$ is described in any case by Theorem
\ref{prult}.
\end{remark}

\section{Conclusions}
In this paper we study the dual minimum distance and minimum weight of codes
arising from quotients of the Hermitian curve through geometric constructions.
We describe the parameters of such codes from a cohomological point of view,
and geometrically characterize the supports of their minimum-weight codewords,
deriving explicit formulas for their number. Our analysis is then applied to
study the dual codes of one-point and two-point Goppa codes on a quotient of
the Hermitian curve.

\begin{acknowledgements}
The authors would like to thank the Referee for useful suggestions and comments
that improved the presentation of this work.
\end{acknowledgements}

\end{document}